\documentclass[11pt]{elsarticle}
\usepackage{epsfig,amsfonts,amssymb,amsmath,amsthm}
\usepackage{color}
\usepackage{array}
\usepackage{lineno}
\usepackage{multirow}
\usepackage{graphicx}
\usepackage{subcaption}
\usepackage{bm}

\usepackage{todonotes}

\makeatletter
\def\ps@pprintTitle{%
  \let\@oddhead\@empty
  \let\@evenhead\@empty
  \let\@oddfoot\@empty
  \let\@evenfoot\@oddfoot
}

\newtheorem{MainTheorem}{Theorem}

\newtheorem{thm}{Theorem}[section]

\newtheorem{cor}[thm]{Corollary}
\newtheorem{lem}[thm]{Lemma}
\newtheorem{prop}[thm]{Proposition}

\newtheorem{rem}[thm]{Remark}

\renewcommand{\epsilon}{\varepsilon}

\makeatother

\begin{document}
\begin{frontmatter}

\title{Monotonous period function for equivariant differential equations with homogeneous nonlinearities}

\author[]{Armengol Gasull}
\ead{Armengol.Gasull@uab.cat}
\author[]{David Rojas\corref{cor1}}
\ead{David.Rojas@uab.cat}

\address{Departament de Matemàtiques, Universitat Autònoma de Barcelona, Edifici C, 08193 Cerdanyola del Vallès (Barcelona), Spain}

\cortext[cor1]{Corresponding author}

%\subjclass[2020]{34C15, 37C27, 70F15}

\begin{abstract}
We prove that the period function of the center at the origin of the $\mathbb{Z}_k$-equivariant differential equation $\dot{z}=iz+a(z\overline{z})^nz^{k+1}, a\ne0,$ is monotonous decreasing for all $n$ and $k$ positive integers, solving a conjecture about them. We show this result as corollary of proving that the period function of the center at the origin of a sub-family of the reversible quadratic centers is monotonous decreasing as well.
\end{abstract}


\end{frontmatter}

%\linenumbers
%%%%%%%%%%%%%%%%%%%%%%%%%%%%%%%%%%%%%%%%%%%%%%%%%%%%%%%%

\section{Introduction}
Real planar autonomous analytic differential systems can be written in complex coordinates in the compact form  $\dot{z}=F(z,\overline{z})$, $z\in\mathbb{C}.$ The $\mathbb{Z}_k$-equivariant differential equations are equations of this form that are invariant under a rotation through the angle $2\pi/k$ about the origin.  They are studied in detail in the classical book of Arnold \cite[Chap. 4 \S 35]{Arnold0} or also in~\cite[Sec. 7]{Li}. For them, the phase portrait on each sector of width $2\pi/k$ centered at the origin is repeated $k$ times. In this paper we are concerned with the simplest family of polynomial $\mathbb{Z}_k$-equivariant differential equations with a non-degenerated center at the origin and homogeneous nonlinearities
\begin{equation}\label{Zk} 
	\dot{z}=iz+a(z\overline{z})^nz^{k+1},
\end{equation}
with $n$ and $k$ a positive integers and $0\ne a\in\mathbb{C}.$ It is also a family of differential equations given by simply two complex monomials, as the ones studied in~\cite{AGP}. The qualitative properties of the period function associated to the origin was posed as an open problem in \cite[Prob. 16]{Gasull}. In particular, it is conjectured that it is monotonous decreasing. The main result of this paper is a positive answer to such conjecture.

\begin{MainTheorem}\label{thmA}
	The period function of the center at the origin of the differential equation~\eqref{Zk}, with $n$ and $k$ positive integers and $0\ne a\in\mathbb{C},$ is monotonous decreasing. Moreover it starts being $2\pi$ at the origin and tends to $2(k+n)\pi/(k+2n)$ when the orbits approach to the boundary of the period annulus. 
\end{MainTheorem}

When $n=0$ or $k=0$ and the origin is a center, the behavior of the period function of the differential equation~\eqref{Zk} is different. It is either constant, decreasing towards $0,$ or increasing towards infinity, see Proposition~\ref{lem:nou}.

As shown in \cite{Gasull}, the period function of the center at the origin of~\eqref{Zk} is strongly related to the period function of the center at the origin of the family of reversible quadratic centers
\begin{equation}\label{loud0}
    \begin{cases}
	\dot{x}=-y+xy,\\
	\dot{y}=x+Dx^2+Fy^2.
	\end{cases} 
\end{equation}
In particular, as we also will see in Section~\ref{sec:ProofThmA}, with the period function for the period annulus of the origin for the sub-family
\begin{equation}\label{loud}
    \begin{cases}
        \dot{x}=-y+xy,\\
        \dot{y}=x+Dx^2+(D+1)y^2,
    \end{cases}    
\end{equation}
with $D\in(-1/2,0)$.

Among the quadratic centers, the period function of the reversible family~\eqref{loud0} is the one with richer qualitative properties. As the counterpart of the renowned conjecture claiming that quadratic systems have at most four limit cycles, it was conjectured in \cite{Chicone2} that quadratic centers have at most two critical periods (that is, number of maxima and minima of the period function). Since then, the qualitative behavior of the period function of the reversible quadratic centers has attracted the attention of the community \cite{MMV0,MMV,MV}. 

In particular, the family~\eqref{loud} with $D\in(-1/2,0)$ was believed to have monotonous decreasing period function, see \cite{MMV0,MMV}.

In fact, a first indication that the period function could be decreasing is given the computation of its first period constant $P_2,$ which gives its local behavior near the origin. More specifically, if ${\bf T}(\rho), \rho>0,$ denotes the minimal period of the orbit starting at $(\rho,0),$ then
\begin{equation}\label{eq:t2}	
	{\bf T}(\rho)=2\pi +P_2\rho^2+O(\rho^3).
\end{equation}	
For system~\eqref{loud0} it is known that
\[P_2= \frac\pi{12}\big(10D^2+10DF-D+4F^2-5F+1\big),
\]	 
see for instance~\cite{MMV}. Hence, when $F=D+1$ some computations give than $P_2=\pi D(2D+1),$ proving that $P_2<0$ when $D\in(-1/2,0)$ and hence that $T(\rho)$ is locally decreasing at the origin. The period annulus of the origin of system~\eqref{loud} is all the half-plane  $\mathcal{B}:=\{(x,y)\,:\,x<1\}.$ It is proved in~\cite[Sec 5]{MMV0} that the period function is also decreasing at the boundary of~$\mathcal B,$ result that is again coherent with the fact that the period function is decreasing.

There exist a large collection of criteria to verify the monotonicity of the period in the literature, \cite{Chicone,GGV,Villadelprat,Schaaf} are some examples. During the last years, the attempts to prove the monotonicity of the period for the family of quadratic centers~\eqref{loud}  with $D\in(-1/2,0)$ have been numerous and unfortunately unfruitful despite all efforts. In particular, to the best of our knowledge, all known criteria have been checked to fail for $D\in(-1,0)\setminus\{-1/2\}$. A similar case where monotonicity holds in system~\eqref{loud0} is when $F=1-2D,$ \cite{Cho}. To see other regions of monotonicity of the period function see~\cite{Villadelprat0} and its references.

Recently, in \cite{Liu} a new adaptation of the criteria introduced in \cite{Villadelprat} has been used to prove the monotonicity of another different but similar family of quadratic centers, the so-called reversible Lotka-Volterra. These are, reversible quadratic centers that belong also to the family of generalized Lotka-Volterra quadratic centers. In particular, they consist on the sub-family $F=-D$. By using this criteria  we prove the following.

\begin{MainTheorem}\label{thmB}
The period function of the center at the origin of system~\eqref{loud} is monotonous decreasing for $D\in(-1/2,0),$ monotonous increasing for $D\in(-\infty,-1/2)\cup(0,+\infty)$ and isochronous  when $D=-1/2$ and $D=0$. Moreover,

\begin{itemize}[(i)]
		
\item If $D\in(-\infty,-1)\cup(0,+\infty)$ the period annulus is the interior of the homoclinic orbit of the equilibrium at $(-1/D,0)$. The period starts being $2\pi$ at the origin and tends to $+\infty$ when the orbits approach the boundary of the period annulus.
\item If $D\in[-1,0]$ the period annulus is all the half-plane $\mathcal{B}=\{(x,y)\,:\,x<1\}$. Moreover, when $D\ne0,$ the period starts being $2\pi$ at the origin and tends\footnote{When $D=-1$ this limit must be understood as infinity.} to $\pi/(D+1),$  when the orbits approach to the boundary of the period annulus.
\end{itemize}

\end{MainTheorem}

As we will see, the asymptotic value of the period function at the boundary of the period annulus is obtained from \cite[Prop. 5.2]{MMV}.  There are results in the Theorem that are already known and we included them for completeness. Indeed, the shape of the period annulus for each value of the parameter is given in \cite{MMV} and the monotonicity for $D\in(-\infty,-1)\cup(0,+\infty)\setminus\{1\}$ was proved in \cite[Thm A]{Villadelprat0}. Moreover, the case $D=1$ was studied in \cite{Cho} as part of the proof of monotonicity for $F=2$. Also, it is well-known that the only isochronous systems in~\eqref{loud0} are
 	\[(D,F)\in\{(-1/2,1/2), (-1/2,2), (0,1/4), (0,1) \},\]
see for instance~\cite{CS}. Notice that the cases $D=-1/2$ and $D=0$ in system~\eqref{loud} precisely correspond to two of them.

In short, our contribution then is the completeness of the study of the monotonicity of the straight line $F=D+1$ to the segment $D\in[-1,0)$, which was unknown to the best of our knowledge (with the exception of the isochronous center). As we have explained, our proof on this segment is based on the general criterion for potential systems commented above. This criterion is introduced in next section and all the rest of the proof is self-contained. Moreover, in Subsection~\ref{se:nova} we study the period function of the other center in family~\eqref{loud} and use another idea that allows to reduce the study of the period function on the above segment to a segment of the family $D=-F,$ providing an alternative proof based totally on the main result of~\cite{Liu}.

The paper is organized as follows. In Section 2 we introduce the monotonicity criterium from~\cite{Liu}. Sections 3 and 4 are devoted to the proofs of Theorems~\ref{thmB} and~\ref{thmA}, respectively.

\section{The monotonicity criterium}
Consider a Hamiltonian system with a Hamiltonian function of the form $H(x,y)={y^2}/{2}+V(x)$ and integrating factor $\ell(x)$, where both $V$ and $\ell$ are analytic functions in a neighborhood of $x=0$. That is, the differential system is written as
\begin{equation}\label{system}
    \dot{x}= -\dfrac{y}{\ell(x)}, \;
    \dot{y}= \dfrac{V'(x)}{\ell(x)}.
\end{equation}
If $V'(0)=0$ and $xV'(x)>0$ for all $x\approx 0$, $x\neq 0$, the origin is a center. The period annulus is the largest punctured neighborhood of the origin foliated by periodic orbits and each periodic orbit of the period annulus can be parametrized by an energy level $h\mapsto\gamma(h)\subset\{(x,y)\in\mathbb{R}^2:H(x,y)=h\}$. The period function of the center assigns to each periodic orbit of the period annulus its minimal period. Due to the previous parameterization and the particular expression of the vector field, the period function can be expressed as an Abelian integral function of the energy,
\begin{equation}\label{period}
T(h):=\int_0^{T(h)}\,\operatorname{d}\!t = \int_{\gamma(h)} \frac{\ell(x)}{y}\,\operatorname{d}\!x.
\end{equation}
Notice that $T(h)={\bf T} \big(\sqrt{2h}\,\big),$ where ${\bf T}$ is given in~\eqref{eq:t2}.

Let us denote by $\mathcal{I}=(x_{\ell},x_r)$ the projection on the $x$-axis of the period annulus of the center at the origin, with $x_{\ell}<0<x_r$. Since $xV'(x)>0$ for all $x\in\mathcal{I}\setminus\{0\}$, the potential function $V$ defines an analytic involution $\sigma$ on $\mathcal{I}$ by
\[
V(\sigma(x))=V(x) \text{ for all }x\in\mathcal{I}\setminus\{0\},
\]
and $\sigma(0)=0$. The following criterium is introduced in \cite[Prop. 2.2]{Liu} as an adaptation of \cite[Thm.~A]{Villadelprat}. 

\begin{prop}\label{prop:Liu}
Let us consider an Abelian integral of the form $I(h)=\int_{\gamma_h}{g(x)}/{y}\,\operatorname{d}\!x$, where $g(x)$ is an analytic function on $\mathcal{I}.$ Define
\[
\mathbf{\Pi}_{\sigma}(g)(x):=\frac{f(x)-f(\sigma(x))\sigma'(x)}{2},\quad\mbox{where}\quad f(x)=-\frac{g(x)}{2}+\Big(\frac{g(x) V(x)}{V'(x)}\Big)'.
\]
If $\mathbf{\Pi}_{\sigma}(g)(x)$ has no zeros in $(0,x_r)$ then $I'(h)$ has no zeros.
\end{prop}

The period function~\eqref{period} is an Abelian integral of the type requested by the Proposition, but the direct application for proving monotonicity is not possible. It can be seen that, by taking $g=\ell$, the corresponding $\mathbf{\Pi}_{\sigma}(\ell)(x)$ changes sign in $(0,x_r)$. As we will show, in order to prove the monotonicity, an auxiliar Abelian integral with a different function $g$ can be used instead.

\section{Study of a family of quadratic reversible systems}

The change of variables given by $\{u=x, v=y(1-x)^{-(D+1)}\}$ transforms system~\eqref{loud} into the equivalent system
\begin{equation}\label{sys2}
    \begin{cases}
        \dot{u}&=-v(1-u)^{D+2},\\
        \dot{v}&=u(1+Du)(1-u)^{-(D+1)},
    \end{cases}    
\end{equation}
with first integral $H(u,v)={v^2}/{2}+V(u)$, where
\[
V(u):=\frac{1}{2}u^2(1-u)^{-2(D+1)}.
\]
Notice that this change is well defined in the whole period annulus of the origin. When $D\in(-1/2,0)$, the projection of the period annulus is $(u_{\ell},u_r)=(-\infty,1)$ and the center is located at $u=0$.
The period function of the center at the origin of~\eqref{sys2} can be written as
\[
T(h)=-\int_{\gamma_h}\frac{(1-u)^{-(D+2)}}{v}\,\operatorname{d}\!u,
\]
where $\gamma_h\subset\{(u,v)\in\mathbb{R}^2:H(u,v)=h\}$ is the closed orbit lying on the level curve $H=h$. Let us define
\begin{equation}\label{I}
	A(h):=-\int_{\gamma_h}(1-u)^{-(D+2)}v\,\operatorname{d}\!u \,\,\text{ and }\,\,
    I(h):=\int_{\gamma_h}\frac{u(1-u)^{-3(D+1)}}{v}\,\operatorname{d}\!u.
\end{equation}

\begin{lem}\label{lem:igualtat}
    For each $h\in(0,+\infty)$, $A'(h)=T(h)$ and $2hT'(h)+\frac{1}{D+1}I'(h)=0$.
\end{lem}
\begin{proof}
The exterior derivative of $\alpha:=(1-u)^{-(D+2)}v$ is $\operatorname{d}\!\alpha=(1-u)^{-(D+2)}\operatorname{d}\!u\wedge\operatorname{d}\!v$ and it admits a Gelfand-Leray form with respect to $H(u,v)$ given by
\[
\frac{\operatorname{d}\!\alpha}{\operatorname{d}\!H}=-\frac{(1-u)^{D+1}}{u(1+Du)}\operatorname{d}\!v.
\]
The derivative of the function $A(h)$ can be expressed as an integral along the same cycle $\gamma_h$ according with formula (3) in page 284 of \cite{Arnold} using the previous Gelfand-Leray form. That is,
\[
A'(h)=-\int_{\gamma_h}\frac{\operatorname{d}\!\alpha}{\operatorname{d}\!H}=\int_{\gamma_h}\frac{(1-u)^{D+1}}{u(1+Du)}\operatorname{d}\!v=-\int_{\gamma_h}\frac{(1-u)^{-(D+2)}}{v}\operatorname{d}\!u=T(h),
\]
where the third equality holds by $\operatorname{d}\!H=V'(u)\operatorname{d}\!u+v\operatorname{d}\!v=0$ on $\gamma_h$. This proves the first assertion of the lemma.

Using the identity $H(u,v)=h$, we have $v=\frac{2}{v}(h-V(u))$ and so
\[
A(h)=-\int_{\gamma_h}\frac{2h(1-u)^{-(D+2)}}{v}\,\operatorname{d}\!u + \int_{\gamma_h}\frac{2(1-u)^{-(D+2)}V(u)}{v}\,\operatorname{d}\!u.
\]
We point out that the first integral of the previous expression is related with $T(h)$. That is,
\begin{equation}\label{eq1}
A(h)=2hT(h) + \int_{\gamma_h}\frac{2(1-u)^{-(D+2)}V(u)}{v}\,\operatorname{d}\!u.
\end{equation}
In addition, integrating by parts,
\[
A(h)=-\int_{\gamma_h}(1-u)^{-(D+2)}v\,\operatorname{d}\!u=-\int_{\gamma_h}v\, \,\operatorname{d}\!\left(\frac{(1-u)^{-(D+1)}}{D+1}\right)=\int_{\gamma_h}\frac{(1-u)^{-(D+1)}}{D+1}\,\operatorname{d}\!v.
\]
Using again $\operatorname{d}\!H=0$ on $\gamma_h$ we can write the previous as
\begin{equation}\label{eq2}
A(h)=-\int_{\gamma_h}\frac{(1-u)^{-(D+1)}V'(u)}{(D+1)v}\,\operatorname{d}\!u.    
\end{equation}
By adding~\eqref{eq1} and~\eqref{eq2}, and using the expression of $V$,
\begin{align*}
   2A(h)&=2hT(h)+\int_{\gamma_h}\left(2(1-u)^{-(D+2)}V(u)-\frac{(1-u)^{-(D+1)}}{D+1}V'(u) \right)\frac{\,\operatorname{d}\!u}{v}\\
   &=2hT(h)+\int_{\gamma_h}\frac{u(1-u)^{-3(D+1)}}{D+1}\frac{\,\operatorname{d}\!u}{v}=2hT(h)+\frac{1}{D+1}I(h).
\end{align*}
Derivating with respect to $h$,
\[
2A'(h)=2T(h)+2hT'(h)+\frac{1}{D+1}I'(h).
\]
Since $A'(h)=T(h)$ by the first assertion of the result, then $2hT'(h)+\frac{1}{D+1}I'(h)=0,$ proving the second assertion of the lemma.
\end{proof}

\begin{proof}[Proof of Theorem~\ref{thmB}]

As mentioned at the introduction, we restrict ourselves to the segment $D\in[-1,0)$, since the monotonicity of the period function for the other parameter values is already known. The particular case $D=-1$ will be treated at the end of the proof because of its simple nature. Moreover, it is already known that $D=-1/2$ corresponds to an isochronous center. Thus, we are left with two open intervals $D\in(-1,-1/2)$ and $D\in(-1/2,0)$. According with the statement of the Theorem, we want to prove that the first of the intervals corresponds to parameters in which the period function is monotonous increasing whereas parameters in the second interval should have monotonous decreasing period. Notice that, once monotonicity is proved, the increasing or decreasing character is determined by the local behavior of the period function near the origin, which is increasing for $D\in(-1,-1/2)$ and decreasing for $D\in(-1/2,0)$ as we shown in~\eqref{eq:t2}. For the sake of brevity, we will focus the proof on the case $D\in(-1/2,0)$. The proof for the other interval is similar enough and we will just give some remarks at the end to highlight the differences.

The strategy for proving the monotonicity of~\eqref{period} is to use Proposition~\ref{prop:Liu} with the Abelian integral~\eqref{I}. The monotonicity then will follow by the second identity in Lemma~\ref{lem:igualtat}. Using the notation in Proposition~\ref{prop:Liu} we have 
\[
I(h)=\int_{\gamma_h}\frac{g(u)}{v}\,\operatorname{d}\!u,\quad \mbox{with}\quad g(u)= u(1-u)^{-3(D+1)},
\]
and
\[
f(u)=-\frac{g(u)}{2}+\Big(\frac{g(u) V(u)}{V'(u)}\Big)'=\frac{u(1-u)^{-3(D+1)}(1+2Du+D(1+2D)u^2)}{2(1+Du)^2}.
\]
Derivating the implicit expression of the involution, $V(\sigma(u))=V(u)$, we have $\sigma'(u)=\frac{V'(u)}{V'(\sigma(u))}$. 
Therefore,
\[
\mathbf{\Pi}_{\sigma}(g)(u)=\frac{f(u)-f(\sigma(u))\sigma'(u)}{2}=\frac{V'(u)}{2}\big(\phi(u)-\phi(\sigma(u))\big)
\]
with $\phi(u)={f(u)}/{V'(u)}$. 
In order to prove that $\mathbf{\Pi}_{\sigma}(g)(u)$ has no zeros on $(0,u_r)=(0,1)$ it is enough to show that the two curves defined by the identities
\begin{equation}\label{original}
V(u)-V(w)=0 \text{ and } \phi(u)-\phi(w)=0
\end{equation}
do not intersect for all $(u,w)\in(0,1)\times(-\infty,0)$. The rest of the section is devoted to prove this fact. 

The first identity of \eqref{original} can be written as
\begin{align*}
V(u)-V(w)&=u^2(1-u)^{-2(D+1)}-w^2(1-w)^{-2(D+1)}\\
&=(u(1-u)^{-(D+1)}+w(1-w)^{-(D+1)})(u(1-u)^{-(D+1)}-w(1-w)^{-(D+1)})=0.
\end{align*}
We note that the second factor does not vanish, since $-\infty<w<0<u<1$. Therefore, the equality $V(u)-V(w)=0$ is equivalent to
\begin{equation}\label{corba1}
F_1(u,w;D):=u(1-u)^{-(D+1)}+w(1-w)^{-(D+1)}=0.
\end{equation}

The second identity of \eqref{original} writes
\[
\frac{(1-u)^{-D}(1+2Du+D(1+2D)u^2)}{(1+Du)^3}-\frac{(1-w)^{-D}(1+2Dw+D(1+2D)w^2)}{(1+Dw)^3}=0.
\]
Since we are interested in $(u,w)$ satisfying both equalities of \eqref{original} to hold, we use~\eqref{corba1} on the previous to obtain an equivalent expression in terms of rational functions,
\begin{equation}\label{corba2}
F_2(u,w;D):=\frac{(1-u)(1+2Du+D(1+2D)u^2)}{u(1+Du)^3}+\frac{(1-w)(1+2Dw+D(1+2D)w^2)}{w(1+Dw)^3}=0. 
\end{equation}
On account of the sign of the polynomial $1+2Dw+D(1+2D)w^2$ we note that~\eqref{corba2} cannot hold for $w<w^*(D)<0$ with
\begin{equation}\label{left}
    w^*(D):=\frac{-D+\sqrt{-D(1+D)}}{D(1+2D)}.
\end{equation}
Thus we can restrict the argument to the bounded region 
\[
\Lambda:=\{(u,w)\in\mathbb{R}^2 : w^*(D)<w<0<u<1\}.
\]
Let $\Gamma_i(D):=\{(u,w)\in\Lambda: F_i(u,w;D)=0\}$. The previous discussion proves that the original curves~\eqref{original} intersect if and only if $\Gamma_1(D)\cap\Gamma_2(D)\neq 0$. We claim that each of the sets $\Gamma_i(D)$ is the graphic of a function and that these two functions do not intersect. Let us state it properly.

\noindent\textbf{Claim.} For each $D\in(-1/2,0)$, $\Gamma_i(D)$, $i=1,2$, defines a curve,
    $\Gamma_i(D)=\{(u,w)\in\Lambda : w=\psi_i(u;D)\}$
    with $\psi_i$ smooth functions and $\psi_1(u;D)\neq\psi_2(u;D)$ for all $u\in(0,1)$.

Notice that the part of the proof of Theorem~\ref{thmB} concerning the monotonicity of the period function follows from the Claim. Let us prove that the Claim is true. We start  showing the assertions concerning $\Gamma_1(D)$. The derivation of $F_1(u,w;D)$ with respect to $w$ yields to
    \[
        \frac{\partial F_1}{\partial w}(u,w;D)=\frac{1+Dw}{(1-w)^{D+2}}
    \]
    which is different from zero for all $w<0$ and $D\in(-1/2,0)$. Therefore, by the implicit function theorem, there exists a smooth function $\psi_1:(0,1)\rightarrow(-\infty,0)$ such that $w=\psi_1(u;D)$ and $F_1(u,\psi_1(u;D);D)=0$. By the derivation with respect to $u$ of the previous, we have
    \[
        \psi_1'(u;D)=-\frac{(1+Du)(1-\psi_1(u;D))^{D}(\psi_1(u;D)-1)^2}{(1+D\psi_1(u;D))(1-u)^{D}(u-1)^2}.
    \]
    Using again $F_1(u,\psi_1(u;D);D)=0$, we also have
    \[
        (1-\psi_1(u;D))^D = - \frac{\psi_1(u;D)(1-u)^{D+1}}{u(1-\psi_1(u;D))}.
    \]
    Thus, the substitution of the latest into the expression of $\psi_1'(u;D)$ implies
    \[
        \psi_1'(u;D)=\frac{\psi_1(u;D)(\psi_1(u;D)-1)(1+Du)}{u(u-1)(1+D\psi_1(u;D))}.
    \]

    The same argument can be followed for $\Gamma_2(D)$, showing that there exists a smooth function $\psi_2:(0,1)\rightarrow(-\infty,0)$ such that $F_2(u,\psi_2(u;D);D)=0$ with
    \[
    \psi_2'(u;D)=-\frac{P(u;D)}{P(\psi_2(u;D);D)}\frac{\psi_2(u;D)^2(1+D\psi_2(u;D))^4}{u^2(1+Du)^4}
    \]
    and $P(x;D)=-1-4Dx-2D(2D-1)x^2-2D(1+D+2D^2)x^3+D^2(1+2D)x^4$. Here we are using that $P(w;D)$ does not vanish for $w^*(D)<w<0$ and $-{1}/{2}<D<0$. Indeed, we have $P(0,D)=-1$ and, since $w^*(D)$ satisfies the quadratic expression $1+2Dw^*(D)+D(1+2D)w^*(D)^2=0$, one can reduce to
    \[
    	P(w^*(D),D)=-\frac{2(D+1)(4D^2+4D-1)w^*(D)+2(2D+3)(D+1)}{(1+2D)^2}<0
    \]
    for all $-{1}/{2}<D<0$. Moreover, the discriminant of $P(w;D)$ with respect to $w$ is 
    \[
    \Delta_w(D):=-16(D+1)^4D^4(304D^4 + 608D^3 + 296D^2 - 8D + 27),
    \]
    which does not vanish for $-{1}/{2}<D<0$. Therefore, the number of zeros of $P(w;D)$ on $(w^*(D),0)$ for any $D$ are the same. In particular, $P(w;-{1}/{3})$ has no zeros on $(w^*(-{1}/{3}),0)$ so $P(w;D)<0$ for all $w^*(D)<w<0$ and $-{1}/{2}<D<0$.
    
    The rest of the proof is to show that $\psi_1(u;D)\neq\psi_2(u;D)$ for any $D\in(-1/2,0)$. With this objective in mind, we first prove that the inequality holds for $D=-1/3$. 

    For $D=-1/3$, the functions $F_1$ and $F_2$ write
    \[
    F_1(u,w;-{1}/{3})=\frac{u}{(1-u)^{\frac{2}{3}}}+\frac{w}{(1-w)^{\frac{2}{3}}},
    \]
    and
    \[
    F_2(u,w;-{1}/{3})=\frac{3Q_2(u,w)}{uw(u-3)^3(w-3)^3},
    \]
    with $Q_2(u,w)=-9u(u-3)^3+3(81-270u+180u^2-36u^3+5u^4)w+(-243+540u-270u^2+18u^3-5u^4)w^2-(u+3)(-27+45u-21u^2+u^3)w^3-(u-1)(-9+6u+u^2)w^4.$ 
    We note that the set $\{(u,w)\in\Lambda:F_1(u,w;-1/3)=0\}$ is the same as the set $\{(u,w)\in\Lambda:Q_1(u,w)=0\}$ with $Q_1(u,w)=u^3(1-w)^2+w^3(1-u)^2$. The resultant
    \[
    \text{Res}(Q_1(u,w),Q_2(u,w),w)=32(u-1)^3u^6R(u)
    \]
    with $R(u)=531441-3188646u+8148762u^2-11455506u^3+9546255u^4-4776408u^5+1487889u^6-406782u^7+143856u^8-32238u^9+1593u^{10}-180u^{11}-4u^{12}.$ If $R(u)$ has no real roots in $(0,1)$ then there are no common zeros of $F_1$ and $F_2$ for $(u,w)$ with $u\in(0,1)$ and so the inequality is proved for $D=-1/3$. We can prove that $R(u)$ has no real roots on $(0,1)$ using Sturm's theorem. For the sake of shortness we do not show the explicit expressions of the Sturm sequence, that can be easily computed with the help of an algebraic manipulator.
    % Let $V(a)$ denote the number of sign changes of the Sturm's sequence at $u=a$. In our situation, $V(0)=6$ and $V(1)=6$. Therefore, Sturm's theorem implies that the number of real roots on $(0,1)$ is $V(0)-V(1)=6-6=0.$ 
    This proves $\psi_1(u)\neq\psi_2(u)$ for $D=-1/3$.   

    Let us finally prove that $\psi_1(u;D)\neq\psi_2(u;D)$ for any $D\in(-1/2,0)$. Assume, with aim of reaching contradiction, that there exist $D^*\in(-1/2,0)$ and $u^*\in(0,1)$ with $\psi_1(u^*;D^*)=\psi_2(u^*;D^*)$. This would imply the existence of a tangent contact between $\psi_1$ and $\psi_2$ for some $(u,D)=(\widehat{u},\widehat{D})$. In consequence, the desired result will follow once we show that
    \[
        \psi_1(u;D)-\psi_2(u;D)=0\;\; \text{ and }\;\; \psi_1'(u;D)-\psi_2'(u;D)=0
    \]
    do not have common solution for any $D\in(-1/2,0)$. Using the expressions of the derivatives computed above, the previous is equivalent to show that
    \[
    F_2(u,w;D)=0\;\; \text{ and }\;\;F_3(u,w;D):=\frac{w(w-1)(1+Du)}{u(u-1)(1+Dw)}+\frac{P(u;D)}{P(w;D)}\frac{w^2(1+Dw)^4}{u^2(1+Du)^4}=0
    \]
    do not intersect on $w^*(D)<w<0<u<1$ for any $D\in(-1/2,0)$. We emphasize that both equalities are given in terms of rational functions. In particular,
    \[
        P_2(u,w;D):=uw(1+Du)^3(1+Dw)^3F_2(u,w;D)
    \]
    and    
    \[
        P_3(u,w;D):=(u-1)u^2(1+Du)^4(1+Dw)P(w;D)F_3(u,w;D)/w,
    \]
    with $P_2(u,w;D)$ a polynomial of degree $7$ and $P_3(u,w;D)$ a polynomial of degree $11$. We omit their expressions for the sake of brevity. Consequently, common zeros of $F_2(u,w;D)=0$ and $F_3(u,w;D)=0$ are equivalent to common zeros of $P_2(u,w;D)=0$ and $P_3(u,w;D)=0$. 

    We now perform the resultant of $P_2(u,w;D)$, $P_3(u,w;D)$ with respect to $w$. That is,
    \[
    \text{Res}(P_2(u,w;D),P_3(u,w;D),w)=S(u;D)R_2(u;D),
    \]
    with 
    \[
    S(u;D)=-8D^9u^7(u-1)^3(1+2D)(D+1)^9(Du+1)^{21}(D(1+2D)u^2+2Du+1),
    \]
    and $R_2(u;D)$ a polynomial of degree $12$ in $u$. The polynomial $S(u;D)$ does not vanish for $u\in(0,1)$ and $D\in(-1/2,0)$. Let us show that $R_2(u;D)$ does not have zeros for $(u,D)\in(0,1)\times(-1/2,0)$ either. We will proceed with a bifurcation argument with respect to the parameter $D$. First, 
    \begin{equation}\label{R201}
    R_2(0;D)=54D(D+1) \text{ and } R_2(1;D)=4D(1+2D)^4(D+1)^9.
    \end{equation}
    Thus no roots enter or escape from the interval $(0,1)$ as we change the value of $D$ within its range. The possible changes in the number of roots of $R_2(u;D)$ in $(0,1)$ with respect to $u$ are then given by the zeros of the discriminant
    \[
    \Delta_u(D):=33554432D^{43}(D+1)^{43}(1+2D)^{32}K_0(D)^3K_1(D)^2K_2(D)^2K_3(D)^2,
    \]
    where $K_0(D):=22D^2+22D+1$, $K_1(D):=128D^4+256D^3+112D^2-16D-3$, and $K_2$ and $K_3$ are polynomials of degree $14$ and $8$, respectively. Let us locate the zeros of $\Delta_u(D)$ on $(-{1}/{2},0)$. $K_0$ and $K_1$ can be solved explicitly and the only roots in $(-{1}/{2},0)$ are 
    \[
    D_0:=-\frac{1}{2}+\frac{3\sqrt{11}}{22} \text{ and }D_1:=-\frac{1}{2}+\frac{\sqrt{5-\sqrt{7}}}{4}.
    \]
    Using again Sturm's theorem we can prove that $K_2$ has a root $D_2\approx -0.1279963$ and $K_3$ has no reals roots on $(-{1}/{2},0)$. The three roots $D_0$, $D_1$ and $D_2$ split the interval $(-1/2,0)$ in four sub-intervals where $R_2(u;D)$ has the same number of roots for any $D$ on each sub-interval. We can then choose any $D$ in each sub-interval to represent the number of zeros of $R_2(u;D)$ in $(0,1)\times(-1/2,0)\setminus\{D=D_i\}_{i=0}^2$ on it. Taking rational representative of $D$ on each sub-interval and using Sturm's theorem we can check that $R_2(u;D)$ has no zeros in $(0,1)$.

	The last part of the proof is devoted to show that $R_2(u;D_i),$ $i=0,1,2,$ has no zeros in $(0,1)$. Sturm's theorem can be directly applied to $R_2(u;D_i)$, $i=0,1$, by using it over the algebraic extension $\mathbb{Q}[D_i]$. For $R_2(u;D_2)$ the same argument can be done, but $D_2$ is implicitly defined by a degree $14$ polynomial, so the computations are more cumbersome. We avoid that by using an alternative argument, which could be also have been applied for the previous two cases. We first notice from the identities~\eqref{R201} that $R_2(u;D_2)$ is negative for $u=0$ and $u=1$. In order to show that $R_2(u;D_2)<0$ for all $u\in(0,1)$ we locate the root $D_2$ inside a suitable rational interval\footnote{These rational upper and lower bounds of $D_2$ are obtained by taking two suitable successive convergents of its continued fraction representation.}
    \begin{align*}
    D_2 &\in[\underline{\delta},\overline{\delta}]:=[-16/125,-267/2086]\approx [-0.128,-0.127996].
    \end{align*}
Indeed, $K_2(\underline{\delta})K_2(\overline{\delta})<0$.
	Thus, we can construct a polynomial $U(u)$ with rational coefficients satisfying $R_2(u;D_2)<U(u)$ by upper-bounding each monomial of $R_2(u;D)$ with $D=\underline{\delta}$ or $D=\overline{\delta}$ depending on the sign of the coefficient and the exponent of $D$. Finally, Sturm's theorem can be used again to show that $U(u)$ is negative in $(0,1)$. Then $R_2(u;D_2)<U(u)<0$ in $(0,1)$. This ends with the proof that $\psi_1(u;D)\neq\psi_2(u;D)$ for any $D\in(-1/2,0)$ and ends the proof of the Claim.  
	
 As we mentioned at the beginning of the proof, the monotonicity of the period function for $D\in(-1,-1/2)$ can be proved in similar fashion. The main difference with respect to the previous is  the treatment  of the system given by equations $P_2(u,w;D)=0$ and $P_3(u,w;D)=0.$ While in the first we start studying it by  computing the resultant of both polynomials with respect to $w,$  in this case, it is more convenient  the compute the resultant  with respect to $u$. We omit the details for the sake of shortness.
	
To finish the proof for $D\in(-1,0)$ it remains to study the behavior of the period function at the origin and at the exterior boundary of the period annulus. That it tends to $2\pi$ at the origin is a well-known fact because it is the period of the linear part of the differential equation, see also~\eqref{eq:t2}. On the other hand, to show that at the exterior boundary is at least $\pi/(D+1)$ it suffices to observe that $x=1$ is an invariant straight line for the flow of equation~\eqref{loud} and that the restricted flow on this line is given by $\dot y= (D+1)(1+y^2).$ Hence, the time $\mathcal{T}$ spend by the flow to travel along all this line is
\begin{equation}\label{eq:t}
\mathcal{T}=\frac1{D+1}\int_{-\infty}^{+\infty} \frac {1}{1+y^2}\,\operatorname{d}\! y=\frac{\pi}{D+1}.
\end{equation}	
By continuous dependence with respect initial conditions the total period of the periodic orbits near the boundary  tends to $\mathcal{T}$ plus the time spend to close the orbit. In fact, in \cite[Prop. 5.2]{MMV0} it is proved that this extra time tends to zero when we approach to the boundary. 

We end the proof with the special case $D=-1$. In this case, system~\eqref{loud} writes
\[
\dot{x}=-y(1-x),\quad
\dot{y}=x(1-x).
\]
It has a straight line $\{x=1\}$ of singularities and $H(x,y)=x^2+y^2$ is a first integral of the system. The period starts being $2\pi$ at the origin as-well, and it tends to infinity as approaching the boundary of the period annulus with the singularity at $(1,0)$. In particular, the period function can be explicitly computed by changing to polar coordinates. Indeed, setting $r^2=x^2+y^2$ and denoting $T(r)$ the minimal period of the orbit with radius $r$, 
\[
T(r)=\int_0^{2\pi} \frac 1{1-r\cos \theta}\,\operatorname{d}\! \theta=\frac{2\pi}{\sqrt{1-r^2}}.
\]
This proves the monotonicity and it ends the proof of the theorem. 
\end{proof}

\subsection{About the period function for the other center of system~\eqref{loud} and an alternative proof of Theorem~\ref{thmB}.}\label{se:nova}
For our interest on equation~\eqref{Zk}, it is not necessary to study the period function of the other center  $(0,-1/D)$  of  system~\eqref{loud}, that exists if and only if $D\in(-1,0).$ Nevertheless, for completeness we devote this section to this question.  As we will see, this approach  provides an alternative proof of Theorem~\ref{thmB}.

Before starting, we remark that when $D\in \mathbb{R}\setminus[-1,0]$ the other critical point $(0,-1/D)$  is a saddle and when $D=-1$ it belongs to a continuum of singularities. 

The key point to study this new period function is next lemma that holds for all quadratic reversible centers~\eqref{loud0}. Its proof follows by simple calculations.

\begin{lem}\label{le:nou} When $D\in(-1,0)$ system~\eqref{loud0} has two centers that are $(0,0)$ and $(-1/D,0).$ Moreover, with the change of variables and time
\[
u= \frac{Dx+1}{D+1},\quad v=\sqrt{\frac{-D}{D+1}}\,y\quad\mbox{and} \quad s=\sqrt{\frac{D+1}{-D}}\,t
\]	
the point $(-1/D,0)$ is translated to $(0,0)$ and the system is transformed into 
\begin{equation*}
	\begin{cases}
		{u}'=-v+uv,\\
		{v}'=u-(D+1)u^2+Fv^2,
	\end{cases} 
\end{equation*}
where the prime corresponds to the derivative with respect to $s.$	
\end{lem}

A straightforward consequence  is:

\begin{cor} When  $D\in(-1,0)$ the period function  of the center $(-1/D,0)$ for system~\eqref{loud0} with the parameters $(D,F)$ behaves as the period function of the origin also for a system of the form~\eqref{loud0} with the same parameter $F$ but where $D$ is replaced by   $-D-1.$
\end{cor}	

Notice that the map $D\mapsto -D-1$ is an involution that leaves the interval $(-1,0)$ invariant.

When we apply the corollary to our particular system~\eqref{loud} with $D\in (-1,0)$ we obtain that if the origin corresponds to $(D,F)=(D,D+1)$ then the point $(-1/D,0)$ corresponds to $(D,F)=(-1-D, D+1)$ which is precisely on the straight line $F=-D,$ that is the one studied in~\cite{Liu}. In short, we can concentrate on system
\begin{equation}\label{loud2}
	\begin{cases}
		\dot{x}=-y+xy,\\
		\dot{y}=x+Dx^2-Dy^2,
	\end{cases}    
\end{equation}
with $D\in(-1,0)$.  For this new system and all values of $D,$ locally near the origin it holds that
\begin{equation*}
	{\bf T}(\rho)=2\pi +P_2\rho^2+O(\rho^3)=2\pi+(2D+1)^2\rho^2+O(\rho^3),
\end{equation*}
where we have used again the general expression of $P_2$ given after~\eqref{eq:t2} when $F=-D.$ Hence, except for $D=-1/2,$ the period function is always locally increasing at the origin. We know that when $D=-1/2$ the origin is an isochronous center. Moreover, by the main result of~\cite{Liu} we also know that otherwise the period function is globally increasing. As in Theorem~\ref{thmB} its limit value at the boundary of the corresponding period annulus can be determined according to the values of $D.$	
Hence, in short, for system~\eqref{loud}, the period function satisfies: 
\begin{itemize}
	\item  When $D\in(-1,-1/2)$ both centers have an increasing period function.
	\item  When $D\in(-1/2,0)$ the origin has a  decreasing period function and the other center $(-1/D,0)$ an increasing period function.
\end{itemize}
Moreover, at the origin it starts at $2\pi$ and at the other center  at $2\pi\sqrt{-D/(D+1)}$ and  it has limit period $\pi/(D+1)$ at the common boundary of both period annuli. In fact notice that
\[
\begin{cases}
	2\pi<2\pi\sqrt{\dfrac{-D}{D+1}}<\dfrac\pi{D+1}, &\quad\mbox{when}\quad D\in (-1,-1/2),\\[0.5cm]
	2\pi\sqrt{\dfrac{-D}{D+1}}<\dfrac\pi{D+1}<2\pi, &\quad\mbox{when}\quad D\in (-1/2,0).
\end{cases}
\] 
Of course, when $D=-1/2$ all four inequalities become equalities.

It is worthwhile to stress that  in~\cite{Liu} the authors also study both centers for system~\eqref{loud2}. By using their results around $(-1/D,0)$ together with Lemma~\ref{lem:nou} we can also deduce the montonicity of the period function of the origin for system~\eqref{loud}. This approach provides an alternative proof to Theorem~\ref{thmB}.

\section{Study of the period function of the $\mathbb{Z}_k$-equivariant differential equation}\label{sec:ProofThmA}

\begin{proof}[Proof of Theorem~\ref{thmA}] We start proving that equation~\eqref{Zk} can be brought to the quadratic system~\eqref{loud}. This fact was already shown in \cite[Prop. 3.1]{Gasull} and we include its proof for completeness.
 
Firstly, since $a\ne0$ it is easy to see that it is not restrictive to reduce the study to the case $a=1.$ Next, we can use polar coordinates $z=re^{i\theta}$ to write equation~\eqref{Zk} with $a=1$ as
\begin{equation}\label{sys_rt}
\frac{\operatorname{d}\! r}{\operatorname{d}\!t} = r^{2n+k+1}\cos(k\theta), \,\, 	\frac{\operatorname{d}\!\theta}{\operatorname{d}\!t}=1+r^{2n+k}\sin(k\theta).
\end{equation}

The change of variables $\{R=r^{2n+k},\Theta=k\theta\}$ and the rescaling of time $\tau=kt$ by the constant factor $k>0$ brings the previous to the system
\[
\frac{\operatorname{d}\!R}{\operatorname{d}\!\tau}=b R^2\cos\Theta, \,\, \frac{\operatorname{d}\!\Theta}{\operatorname{d}\!\tau} = 1+ R\sin\Theta,
\]
where $b:=1+{2n}/{k}$. Rewriting the system in cartesian coordinates $X+iY=Re^{i\Theta}$, we obtain
\begin{equation}\label{sys_XY}
\frac{\operatorname{d}\!X}{\operatorname{d}\!\tau}=-Y+bX^2-Y^2, \, \, \frac{\operatorname{d}\!Y}{\operatorname{d}\!\tau}=X+(1+b)XY.
\end{equation}
Now, using the change of variables given by $\{x=-(1+b)Y,y=-(1+b)X\}$ and changing the sign of time $s=-\tau$ we arrive to
\begin{equation}\label{sys_xy}
\frac{\operatorname{d}\!x}{\operatorname{d}\!s} = -y+xy, \, \, \frac{\operatorname{d}\!y}{\operatorname{d}\!s} = x-\frac{1}{1+b}x^2+\frac{b}{1+b}y^2,
\end{equation}
which is system~\eqref{loud} with  $D=-1/(1+b)=-k/(2(k+n))\in(-1/2,0).$

As we can see from the periodicity of system~\eqref{sys_rt} with respect to $\theta$, system \eqref{Zk} is $\mathbb{Z}_k$-equivariant. In particular, the period of a periodic orbit of~\eqref{Zk} around the origin is $k$ times the time spent by the orbit of \eqref{sys_rt} from  $\theta=0$ to $\theta={2\pi}/{k}.$ By the construction of system~\eqref{sys_XY}, this time is exactly $1/k$ times the period of the corresponding periodic orbit around the origin of~\eqref{sys_XY}, which equals the period of the periodic orbit around the origin of~\eqref{sys_xy}. Thus, we have that the period function at the origin of system~\eqref{Zk}  equals the period function at the origin of system~\eqref{loud} with  $D=-k/(2(k+n))\in(-1/2,0).$

Therefore, by Theorem~\ref{thmB}, the period function of the center at the origin of the quadratic reversible center~\eqref{loud} when $D\in(-1/2,0$) is monotonous decreasing. So, we also have that the period function of the center at the origin of equation~\eqref{Zk} is monotonous decreasing, proving the first part of the result.
 
The period function at the origin tends to $2\pi.$ The proof is the same that in Theorem~\ref{thmB}. The proof at the boundary of the period annulus follows from the similar result proved in Theorem~\ref{thmB}. Indeed, from the previous discussion the period function of the center at the origin of~\eqref{Zk} equals the period function at the origin of~\eqref{loud} with $D=-k/(2(k+n)).$ In particular, they have the same limit when approaching the boundary of the period annulus, which is given by
\[
	\mathcal{T} = \frac \pi{D+1}=\frac{2(k+n)\pi}{k+2n},
\]
where  $\mathcal{T}$ is introduced in~\eqref{eq:t}. 
\end{proof}

For completeness, we end this section by studying the period function for the cases $n=0$ or $k=0$ in the differential equation~\eqref{Zk}.

\begin{prop}\label{lem:nou} (i) The origin of the differential equation~\eqref{Zk} with $n=0$ is an isochronus center.

(ii) The origin of the differential equation~\eqref{Zk} with $k=0$ is a center if and only if $\operatorname{Re}(a)=0.$ Moreover, when $a=\alpha i,$ $\alpha\in\mathbb{R},$ the function $V(z,\overline z)=z\overline z$ is a first integral of the equation and the period function of the origin parameterized by $u=z\overline z$ is
\begin{equation}\label{eq:per}
	T(u)=\frac{2\pi}{1+\alpha u^n}.
\end{equation}
In particular, when $\alpha>0$ the origin is a global center with decreasing period function that tends to zero at infinity, and when $\alpha<0$ the period annulus of the origin is $\{z\overline z\le (-1/\alpha)^{1/n}\},$ its boundary is full of equilibria and the period function is increasing and tends to infinity at this boundary.	
\end{prop}
	
\begin{proof}
(i) When $n=0$ the differential equation is holomorphic and the origin is an isochronous center, see~\cite{CS} or the references therein.
		
(ii) When $k=0$ consider the Lyapunov function $V(z,\overline z)=z\overline z.$ 
Then $\dot V(z,\overline z)=2 \operatorname{Re}(a)(z\overline z)^{n+1}.$ Hence the origin is a center if and only if  $\operatorname{Re}(a)=0$ as we wanted to prove. When $a=\alpha i$ the expression of the differential equation is
\[
	\dot{z}=iz\big(1+\alpha(z\overline{z})^n),
\]
where notice that $1+\alpha(z\overline{z})^n=1+\alpha u^n$ takes only real values and so it provides a reparametrization of the trajectories, that are circles centred at the origin. Writing it in polar coordinated we obtain the expression~\eqref{eq:per}. All the properties of $T(u)$ described in the statement follow by studying its graph.
\end{proof}	

\begin{rem}
Notice that in the proof of Theorem~\ref{thmA} it holds that $D=-1/2$ if and only if $n=0$. This corresponds to the isochronous case studied in the above proposition.
\end{rem}


\begin{thebibliography}{99}
	
\bibitem{AGP} M.J. Álvarez, A. Gasull, R. Prohens. Uniqueness of the limit cycles for complex differential equations with two monomials. {\it J. Math. Anal. Appl.} {\bf 518},  (2023) 126663.
	
	
\bibitem{Arnold0}
V.I. Arnol'd. {\it Chapitres supplémentaires de la théorie des équations différentielles ordinaires}  (in French). Editions Mir-Moscou, 1980.

	
\bibitem{Arnold}
V.I. Arnol'd, S.M. Guse\v{\i}n-Zade, A.N. Varshenko. {\it Singularities of differentiable maps. Vol II. Monodromy and asymptotics of integrals.} 
%Translated from the Russion by Hugh Porteous. Translation revised by the authors and James Montaldi. 
Monographs in Mathematics, \textbf{83}. Birkhäuser Boston, Inc., Boston, MA, 1988.



\bibitem{CS}
J. Chavarriga, M. Sabatini. A survey of isochronous centers. {\it Qual. Theory Dyn. Syst.} {\bf 1}(1), (1999) 1--70.

\bibitem{Chicone}
C. Chicone. The monotonicity of the period function for planar Hamiltonian vector fields. {\it J. Differential Equations} {\bf 69} (1987) 310--321.

\bibitem{Chicone2}
C. Chicone. Review in MathSciNet, Ref. 94h:58072.



\bibitem{Cho} R. Chouikha. Monotonicity of the period function for some planar differential
systems. Part I: Conservative and quadratic systems. {\it Appl. Math. (Warsaw)} {\bf 32}
(2005) 305--325.

\bibitem{Gasull}
A. Gasull. Some open problems in low dimensional dynamical systems. {\it SeMA Journal} {\bf 78} (2021) 233--269.

\bibitem{GGV}
A. Gasull, A. Guillamon, J. Villadelprat. The period function for second-order quadratic ODEs is monotone. {\it Qual. Theo. Dyn. Syst.} {\bf 4} (2004)  329--352.

\bibitem{Li}
J. Li. Hilbert's 16th problem and bifurcations of planar polynomial vector fields. {\it Int. J. Bifurc. Chaos Appl. Sci. Eng.} {\bf 13}(1) (2003) 47--106.

\bibitem{Liu}
J. Li, C. Li, C. Liu, D. Wang. The period function of reversible Lotka-Volterra quadratic centers. {\it J. Differential Equations} {\bf 307} (2022) 556--579.

\bibitem{Villadelprat}
F. Mañosas, J. Villadelprat. Criteria to bound the number of critical periods. {\it J. Differential Equations} {\bf 246} (2009) 2415--2433.


\bibitem{MMV0}
P. Marde\v{s}ić, D. Marín, J. Villadelprat. On the time function of the Dulac map for families
of meromorphic vector fields. {\it Nonlinearity} {\bf 16} (2003) 855.

\bibitem{MMV}
P. Marde\v{s}ić, D. Marín, J. Villadelprat. The period function of reversible quadratic centers. {\it J. Differ. Equ.} {\bf 224}(1) (2006) 120--171.

\bibitem{MV}
D. Marín, J. Villadelprat. The criticality of reversible quadratic centers at the
outer boundary of its period annulus. {\it J. Differ. Equ.} {\bf 332} (2022) 123--201.

\bibitem{Schaaf}
R. Schaaf. A class of Hamiltonian systems with increasing periods. {\it J. Reine Angew. Math.} {\bf 363} (1985) 96--109.

\bibitem{Villadelprat0} J. Villadelprat.
On the reversible quadratic centers with monotonic period function.
{\it Proc. Amer. Math. Soc.} {\bf 135}(8) (2007) 2555--2565. 

\end{thebibliography}
\end{document}